\documentclass[12pt]{amsart}

\usepackage{amssymb,latexsym}

\usepackage{enumerate}

\makeatletter

\@namedef{subjclassname@2010}{

  \textup{2010} Mathematics Subject Classification}

\makeatother
\newtheorem{thm}{Theorem}[section]
\newtheorem{Con}[thm]{Conjecture}

\newtheorem{cor}[thm]{Corollary}
\newtheorem{lem}[thm]{Lemma}
\newtheorem{pro}[thm]{Proposition}
\theoremstyle{definition}

\numberwithin{equation}{section}

\newcommand{\re}{\textup{Re}}
\newcommand{\im}{\textup{Im}}

\begin{document}

\baselineskip=17pt

\title{Extreme values of class numbers of real quadratic fields}

\author[Youness Lamzouri]{Youness Lamzouri}

\address{Department of Mathematics and Statistics,
York University,
4700 Keele Street,
Toronto, ON,
M3J1P3
Canada}

\email{lamzouri@mathstat.yorku.ca}

\date{}

\begin{abstract}  
We improve a result of H. L. Montgomery and J. P. Weinberger by establishing the existence of infinitely many fundamental discriminants $d>0$ for which the class number of the real quadratic field $\mathbb{Q}(\sqrt{d})$ exeeds $(2e^{\gamma}+o(1)) \sqrt{d}(\log\log d)/\log d$. We believe this bound to be best possible. We also obtain upper and lower bounds of nearly the same order of magnitude, for the number of real quadratic fields with discriminant $d\leq x$ which have such an extreme class number.
\end{abstract}


\maketitle

\section{Introduction}

An important problem in number theory is to understand the size of the class number of an algebraic number field. The case of a quadratic field has a long history going back to Gauss.  Let $d$ be a fundamental discriminant and $h(d)$ be the class number of the field
$\mathbb{Q}(\sqrt{d})$. When $d<0$, in which case $\mathbb{Q}(\sqrt{d})$ is imaginary quadratic,  J. E. Littlewood \cite{Li} established, assuming the Generalized Riemann Hypothesis GRH, that 
\begin{equation}\label{BoImGRH}
h(d)\leq \left(\frac{2e^{\gamma}}{\pi}+o(1)\right) \sqrt{|d|}\log\log |d|,
\end{equation}
where $\gamma$ is the Euler-Mascheroni constant.
Littlewood used Dirichlet's class number formula,  which for $d<-4$, asserts that 
\begin{equation}\label{Dirichlet}
h(d)=\frac{\sqrt{|d|}}{\pi}\cdot L(1, \chi_d), 
\end{equation} 
where $\chi_d=\left(\frac{d}{\cdot}\right)$ is the Kronecker symbol. He then deduced \eqref{BoImGRH} from the bound
\begin{equation}\label{BoL1GRH}
L(1,\chi_d)\leq (2e^{\gamma}+o(1))\log\log |d|,
\end{equation}
which he obtained under GRH for all fundamental discriminants $d$. Furthermore, under the same hypothesis,  Littlewood \cite{Li} proved that there exist infinitely many fundamental discriminants $d$ (both positive and negative) for which 
\begin{equation}\label{OmegaL1}
L(1,\chi_d)\geq (e^{\gamma}+o(1))\log\log |d|,
\end{equation}
and hence for those $d<0$, one has
\begin{equation}\label{OmegaImaginary}
 h(d)\geq \left(\frac{e^{\gamma}}{\pi}+o(1)\right) \sqrt{|d|}\log\log |d|,
\end{equation}
by the class number formula \eqref{Dirichlet}.
The omega result \eqref{OmegaL1} was later established unconditionally by S. Chowla \cite{Ch}. 

The case of a real quadratic field is notoriously  difficult,  due to the presence of non-trivial units in $\mathbb{Q}(\sqrt{d})$. For example,  Gauss's conjecture that there are infinitely many positive discriminants $d$ for which $h(d)=1$ is still open. When $d$ is positive, the class number $h(d)$ is heavily affected by the size of the regulator $R_d$ of $\mathbb{Q}(\sqrt{d})$. In this case,  Dirichlet's class number formula asserts that  
\begin{equation}\label{ClassReal}
h(d)= \frac{\sqrt{d}}{R_d}\cdot L(1, \chi_d).
\end{equation}
Recall that $R_d=\log \varepsilon_d$ where $\varepsilon_d$ is the fundamental unit of the quadratic field  $\mathbb{Q}(\sqrt{d})$, defined as  
$\varepsilon_d= (a+b\sqrt{d})/2,$
where $b>0$ and $a$ is the smallest positive integer such that $(a,b)$ is a solution to the Pell equations $m^2-dn^2=\pm 4$. Since $\varepsilon_d>\sqrt{d}/2$, it follows that when $d$ is large we have 
\begin{equation}\label{LowerRegulator}
R_d\geq \left(\frac12+o(1)\right) \log d,
\end{equation} and hence by Littlewood's bound \eqref{BoL1GRH} we deduce that
\begin{equation}\label{}
h(d) \leq \big(4 e^{\gamma} +o(1)\big) \sqrt{d}\cdot \frac{\log\log d}{\log d}.
\end{equation}
for all positive fundamental discriminants $d$, under the assumption of GRH. In 1977, H. L. Montgomery and J. P. Weinberger \cite{MoWe} showed that this bound cannot be improved, apart from the value of the constant. Indeed, they proved that there exist infinitely many real quadratic fields $\mathbb{Q}(\sqrt{d})$ such that 
\begin{equation}\label{MW1977}
h(d)\gg \sqrt{d}\cdot \frac{\log\log d}{\log d}.
\end{equation}
Recently, W. Duke investigated generalizations of this result to higher degree number fields. In \cite{Du2}, he obtained the corresponding  omega result for the class number of abelian cubic fields, while in \cite{Du1}, assuming certain hypotheses (including GRH), he obtained similar results in the case of totally real number fields of a fixed degree whose normal closure has the symmetric group as Galois group.

It is widely believed that the true nature of extreme values of $L(1,\chi_d)$ is given by the omega result \eqref{OmegaL1} rather than the GRH bound \eqref{BoL1GRH}. A. Granville and K. Soundararajan \cite{GrSo2} investigated the distribution of large values of $L(1, \chi_d)$ and their results give strong support to this conjecture. In view of \eqref{LowerRegulator} and the class number formula \eqref{ClassReal}, this leads to the following conjecture
\begin{Con}\label{max}
For all large positive fundamental discriminants $d$ we have 
\begin{equation}\label{ConjBound}
h(d)\leq \big(2 e^{\gamma} +o(1)\big) \sqrt{d}\cdot \frac{\log\log d}{\log d}.
\end{equation}
\end{Con}
In this paper, we prove the existence of infinitely many real quadratic fields $\mathbb{Q}(\sqrt{d})$ for which  the class number $h(d)$ is as large as the conjectured upper bound \eqref{ConjBound}. We also obtain upper and lower bounds of nearly the same order of magnitude,  for the number of real quadratic fields with discriminant $d\leq x$ for which the class number is that large.
\begin{thm}\label{Main}
Let $x$ be large.
\begin{itemize}
\item[(a)] There are at least $x^{1/2-1/\log\log x}$ real quadratic fields $\mathbb{Q}(\sqrt{d})$ with discriminant  $d\leq x$, such that 
\begin{equation}\label{Extreme}
h(d)\geq  (2e^{\gamma}+o(1))\sqrt{d}\cdot \frac{\log\log d}{\log d}.
\end{equation}
\item[(b)]
Furthermore, there are at most $x^{1/2+o(1)}$ real quadratic fields $\mathbb{Q}(\sqrt{d})$ with discriminant $d\leq x$, for which \eqref{Extreme} holds. 
\end{itemize}
\end{thm}
To prove the omega result \eqref{MW1977}, Montgomery and Weinberger worked over the following special family of fundamental discriminants, first studied by Chowla
$$ \mathcal{D}:=\{ d \text{ square-free of the form } d= 4n^2+1 \text{ for some } n\geq 1\}.$$
This family has the advantage that the regulator $R_d$ is as small as possible. Indeed, if $d=4n^2+1\in \mathcal{D}$, then the fundamental unit of $\mathbb{Q}(\sqrt{d})$ is 
$\epsilon_d= 2n+\sqrt{d}\leq 2 \sqrt{d}$, and hence $R_d= (1/2+o(1))\log d$. Therefore, the class number formula \eqref{ClassReal} implies that
\begin{equation}\label{ClassSpecial}
h(d)= (2+o(1))\frac{\sqrt{d}}{\log d} \cdot L(1, \chi_d)
\end{equation}
for $d\in \mathcal{D}$. Montgomery and Weinberger showed that there exist infinitely many $d\in \mathcal{D}$ for which $L(1,\chi_d)\gg \log\log d$, from which they deduced \eqref{MW1977}. 

Let $\mathcal{D}(x)$ denote the set of $d\in \mathcal{D}$ such that $d\leq x$. To establish the first part of Theorem \ref{Main}, we prove that there exist many fundamental discriminants  $d\in \mathcal{D}(x)$ for which $L(1,\chi_d)$ is as large as one could hope for, namely $\geq (e^{\gamma}+o(1))\log\log d$. Our argument uses elements from the work of Montgomery and Weinberger as well as some new ideas. 

The proof of the second part of  Theorem \ref{Main} relies on two main ingredients. First, using elementary methods we bound the number of real quadratic fields $\mathbb{Q}(\sqrt{d})$ with discriminant $d\leq x$ which have small regulator. Then, we combine Heath-Brown's quadratic large sieve \cite{HB}   with zero-density estimates to show that  Littlewood's GRH bound $L(1,\chi_d)\leq (2e^{\gamma}+o(1))\log\log d$ holds unconditionally for all but at most $x^{1/2+o(1)}$ fundamental discriminants $0<d<x$.

\section{Real quadratic fields with extreme class number: proof of Theorem \ref{Main}, part (a)}

To obtain large values of $L(1, \chi_d)$, a general strategy is to construct fundamental discriminants $d$ for which $\chi_d(p)=1$ for all the small primes $p$, typically up to $y=\log d$.  Montgomery and Weinberger \cite{MoWe} noticed that for square-free $d$ of the form $d=4n^2+1$, one has $\chi_d(p)=1$ for all the primes $p$ dividing $n$. Hence, this reduces the problem to estimating the number of  fundamental discriminants $d\in \mathcal{D}(x)$ for which $(d-1)/4$ is divisible by all the small primes. To this end they established the following lemma.
\begin{lem}[Lemma 1 of \cite{MoWe}]\label{MontgomeryWeinberger}
The number of  integers $d\leq x$ such that $d$ is square-free and $d=4n^2+1$ where $q \mid n$ equals
$$
\frac{\sqrt{x}}{2q}\prod_{p\nmid q}\left(1-\frac{2}{p^2}\right)+O\left(x^{1/3}\log x\right).
$$
\end{lem}
Taking $q=\prod_{p\leq y} p$, where $\sqrt{\log x}\leq y\leq (\log x)/8$ is a real number, and noting that $q=e^{y(1+o(1))}$ by the prime number theorem, yields 
\begin{cor}\label{MoWe}
Let $\sqrt{\log x}\leq y\leq (\log x)/8$ be a real number. The number of  fundamental discriminants  $d\in \mathcal{D}(x)$ such that $\chi_d(p)=1$ for all primes $p\leq y$ is at least
$
x^{1/2}e^{-y(1+o(1))}.
$
\end{cor}
Montgomery and Weinberger then used zero density estimates to prove that for any $0<\delta<1$,  all but at most $x^{\delta}$ fundamental discriminants $1\leq d\leq x$ satisfy 
\begin{equation}\label{ZDShort}
\log L(1, \chi_d)=\sum_{p\leq y} \frac{\chi_d(p)}{p}+O_{\delta}(1),
\end{equation}
where $(\log x)^{\delta}\leq y\leq \log x$ is a real number. Taking $y= (\log x)/9$ and $\delta=1/4$ in \eqref{ZDShort}, and using Corollary \ref{MoWe} produces more than $x^{3/8}$ fundamental discriminants $d\leq x$ in $\mathcal{D}$ for which $L(1,\chi_d)\gg \log\log d.$

In order to improve this estimate, we first replace \eqref{ZDShort} with a better approximation to $L(1, \chi_d)$, due to Granville and Soundararajan  \cite{GrSo2}, which is obtained using zero density estimates together with the large sieve.
 \begin{pro}[Proposition 2.2 of \cite{GrSo2}]\label{GranvilleSound}
Let $A>2$ be fixed. Then, for all but at most $Q^{2/A+o(1)}$ primitive characters $\chi \pmod q$ with $q\leq Q$ we have 
$$ 
L(1,\chi)=\prod_{p\leq (\log Q)^A}\left(1-\frac{\chi(p)}{p}\right)^{-1}\left(1+O\left(\frac{1}{\log\log Q}\right)\right).
$$
\end{pro}

 Let $\sqrt{\log x}\leq y \leq  (\log x)/8$ be a real number. Then, note that 
\begin{equation}\label{FirstCond}
\prod_{p\leq (\log x)^A}\left(1-\frac{\chi_d(p)}{p}\right)^{-1}= \prod_{p\leq y}\left(1-\frac{\chi_d(p)}{p}\right)^{-1}\exp\left(\sum_{y<p<(\log x)^A}\frac{\chi_d(p)}{p} +O\left(\frac{1}{\sqrt{y}\log y}\right)\right).
\end{equation}
By Corollary \ref{MoWe}, there are many fundamental discriminants $d\in \mathcal{D}(x)$ for which the product  $\prod_{p\leq y}(1-\chi_d(p)/p)^{-1}$ is as large as possible. The key ingredient in the proof of  the first part of Theorem \ref{Main} is the following proposition which gives an upper bound for the $2k$-th moment of $\sum_{y<p<(\log x)^A}\chi_d(p)/p$ as $d$ varies in  $\mathcal{D}(x)$, uniformly for $k$ in a large range. In particular, we shall later deduce that with very few exceptions in $\mathcal{D}(x)$, the prime sum $\sum_{y<p<(\log x)^A}\chi_d(p)/p$ is small. 
\begin{pro}\label{LargeMoments}
Let $\sqrt{\log x}< y< \log x$ be a real number, and  $z=(\log x)^A$ where $A>2$ is a constant.  Then, for every positive integer $k\leq \log x/(8A\log\log x)$ we have 
$$\sum_{d \in \mathcal{D}(x)}\Bigg(\sum_{y< p< z}\frac{\chi_d(p)}{p}\Bigg)^{2k}\ll \sqrt{x} \left(\frac{ck}{y\log y}\right)^{k},$$
for some absolute constant $c>0$.
\end{pro}
To prove this result we first need the following lemma, which gives a non-trivial bound for a certain character sum.
\begin{lem}\label{Key}
Let $q$ be an odd positive integer, and write $q=q_1^2q_0$ where $q_0$ is square-free.  If $x\geq q^2$ then
$$\sum_{n\leq x} \left(\frac{4n^2+1}{q}\right)\ll \frac{x}{q_0},$$ 
where $(\frac{\cdot}{q})$ is the Jacobi symbol modulo $q$. 
\end{lem}
\begin{proof}
Let $f(n)=4n^2+1$. First we have 
\begin{equation}\label{reduction}
\sum_{n\leq x} \left(\frac{f(n)}{q}\right) = \sum_{a=1}^q\sum_{\substack{n\leq x\\ n\equiv a \bmod q}}\left(\frac{f(n)}{q}\right)= \sum_{a=1}^q\left(\frac{f(a)}{q}\right)
\sum_{\substack{n\leq x\\ n\equiv a \bmod q}}1= \frac{x}{q} \sum_{a=1}^q\left(\frac{f(a)}{q}\right) + O(q).
\end{equation}
Note that $\sum_{a=1}^q\left(\frac{f(a)}{q}\right)$ is a complete character sum, and hence if $q=p_1^{\alpha_1}\cdots p_k^{\alpha_k}$ is the prime factorization of $q$, then 
\begin{equation}\label{multiplicative}
\sum_{a=1}^q\left(\frac{f(a)}{q}\right)=\prod_{j=1}^k \left(\sum_{a_j=1}^{p_j^{\alpha_j}}\left(\frac{f(a_j)}{p_j^{\alpha_j}}\right)\right)= \prod_{j=1}^k \left(\sum_{a_j=1}^{p_j^{\alpha_j}}\left(\frac{f(a_j)}{p_j}\right)^{\alpha_j}\right),
\end{equation}
by multiplicativity and the Chinese Remainder theorem. Now, if $\alpha_j=2\beta_j$ is even we use the trivial bound
\begin{equation}\label{even}
\left|\sum_{a=1}^{p_j^{2\beta_j}}\left(\frac{f(a)}{p_j}\right)^{2\beta_j}\right|\leq p_j^{2\beta_j}.
\end{equation}
On the other hand, if $\alpha_j=2\beta_j+1$ is odd, then $\left(\frac{f(a_j)}{p_j}\right)^{\alpha_j}=\left(\frac{f(a_j)}{p_j}\right)$, and hence  
\begin{equation}\label{odd1}
\sum_{a=1}^{p_j^{2\beta_j+1}}\left(\frac{f(a)}{p_j}\right)^{2\beta_j+1}= \sum_{b=0}^{p_j^{2\beta_j}-1}\sum_{c=1}^{p_j}\left(\frac{f(bp_j+c)}{p_j}\right)=p_j^{2\beta_j}\sum_{c=1}^{p_j}\left(\frac{f(c)}{p_j}\right).
\end{equation}
If $p$ is a prime number,  then the sum $\sum_{n=1}^p \left(\frac{n^2+b}{p}\right)$ is a Jacobsthal sum, and it is known that  (see for example Storer \cite{St})
$$  \sum_{n=1}^p \left(\frac{n^2+b}{p}\right) =-1 \textup{ if } p\nmid b.$$
Therefore, we deduce
$$
\sum_{c=1}^{p_j}\left(\frac{4c^2+1}{p_j}\right)=-\left(\frac{4}{p_j}\right)=-1.
 $$
Inserting this estimate in \eqref{odd1} yields
\begin{equation}\label{odd2}
 \left|\sum_{a=1}^{p_j^{2\beta_j+1}}\left(\frac{f(a)}{p_j}\right)^{2\beta_j+1}\right|= p_j^{2\beta_j}.
 \end{equation}
Combining the bounds \eqref{even} and \eqref{odd2} in \eqref{multiplicative} yields 
$$ \left|\sum_{a=1}^q\left(\frac{f(a)}{q}\right)\right|\leq \frac{q}{q_0}. $$
The result follows upon inserting this bound in \eqref{reduction}.
\end{proof}

\begin{proof}[Proof of Proposition \ref{LargeMoments}]
As before, we let $f(n)=4n^2+1$.  By positivity of the summand we have
$$ 
\sum_{d \in \mathcal{D}(x)}\Bigg(\sum_{y< p< z}\frac{\chi_d(p)}{p}\Bigg)^{2k}\leq \sum_{n\leq \sqrt{x}}
\left(\sum_{y< p< z}\frac{\left(\frac{f(n)}{p}\right)}{p}\right)^{2k}.
$$
Expending the inner sum we obtain
\begin{equation}\label{Expansion}
\left(\sum_{y< p< z}\frac{\left(\frac{f(n)}{p}\right)}{p}\right)^{2k}= 
\sum_{m}\frac{b_{2k}(m;y, z)\left(\frac{f(n)}{m}\right)}{m}, 
\end{equation}
where 
\begin{equation}\label{Combinatorics}
b_{r}(m;y, z):= \sum_{\substack{y<p_1, \dots, p_{r}<z\\ p_1\cdots p_{r}=m}}1.
\end{equation}
Note that $0\leq b_{r}(m;y,z)\leq r!$. Moreover, $b_{r}(m;y,z)=0$ unless $m=p_1^{\alpha_1}\cdots p_{s}^{\alpha_s}$ where  $y< p_1<p_2<\cdots<p_s< z$ are distinct primes and $\Omega(m)=\alpha_1+\cdots+\alpha_s=r$ (where $\Omega(m)$ is the number of prime divisors of $m$ counting multiplicities). In this case, we have 
\begin{equation}\label{CombFormula}
b_{r}(m;y, z)=\binom{r}{\alpha_1, \dots, \alpha_s}.
\end{equation}
Using this formula, one can easily deduce that if $n$ and $m$ are positive integers with $\Omega(n)=\ell$ and $\Omega(m)=r$ then
\begin{equation}\label{CombInequality}
b_{\ell+r}(mn;y,z)\leq  \binom{\ell+r}{\ell} b_{\ell}(n;y,z)b_{r}(m;y,z).
\end{equation}
Since $z^{2k}\leq x^{1/4}$, then it follows from Lemma \ref{Key} that
\begin{equation}\label{bound}
\begin{aligned}
 \sum_{n\leq  \sqrt{x}}\left(\sum_{y< p< z}\frac{\left(\frac{f(n)}{p}\right)}{p}\right)^{2k}
 &=\sum_{m}\frac{b_{2k}(m;y, z)}{m} \sum_{n\leq \sqrt{x}}\left(\frac{f(n)}{m}\right)
 \ll   \sqrt{x} \sum_{m}\frac{b_{2k}(m;y, z)}{mm_0},
 \end{aligned}
 \end{equation}
 where $m_0$ is the square-free part of $m$.  Let $m$ be such that $\Omega(m)=2k$ and $p\mid m \implies y<p<z$, and write $m=m_1^2m_0$, where $m_0$ is square-free. Put $\Omega(m_1)=\ell$. Then by \eqref{CombInequality} we have 
 \begin{align*}
 b_{2k}(m;y, z) &\leq \binom{2k}{2\ell} b_{2\ell}(m_1^2;y,z)b_{2k-2\ell}(m_0;y,z)\\
 &\leq \binom{2k}{2\ell}\binom{2\ell}{\ell} \left(b_{\ell}(m_1;y,z)\right)^2b_{2k-2\ell}(m_0;y,z)\\
 &\leq \frac{(2k)!}{k!} \binom{k}{\ell}b_{\ell}(m_1;y,z)b_{2k-2\ell}(m_0;y,z),
\end{align*}
since $b_{\ell}(m_1;y,z)\leq \ell!$. 
Therefore, we deduce that 
\begin{align*}
\sum_{m}\frac{b_{2k}(m; y, z)}{mm_0} &\leq \frac{(2k)!}{k!}\sum_{\ell=0}^k\binom{k}{\ell}\sum_{m_1} \frac{b_{\ell}(m_1;y,z)}{m_1^2}\sum_{m_0}\frac{b_{2k-2\ell}(m_0;y,z)}{m_0^2}\\
&\leq\frac{2^k(2k)!}{k!} \left(\sum_{y<p<z}\frac{1}{p^2}\right)^{k} \leq \left(\frac{ck}{y\log y}\right)^k,
\end{align*}
for some positive constant $c>0$ if $y$ is large enough,
since 
\begin{equation}\label{PrimeSumPower}
\sum_{n}\frac{b_{r}(n;y, z)}{n^2}= \left(\sum_{y<p<z}\frac{1}{p^2}\right)^r
\end{equation}
and
$
\sum_{y<p<z} 1/p^2\ll 1/(y\log y)
$
 by the prime number theorem. Inserting this bound in \eqref{bound} completes the proof.
\end{proof}

We are now ready to prove the first part of Theorem \ref{Main}.

\begin{proof}[Proof of Theorem \ref{Main}, part (a)]

Let $z= (\log x)^6$, and $\sqrt{\log x}\leq y \leq (\log x)/8$ be a real number to be chosen later. Then, by Proposition \ref{GranvilleSound} and equation \eqref{FirstCond} it follows that for all but at most $x^{2/5}$ fundamental discriminants $d\in \mathcal{D}(x)$, we have
\begin{equation}\label{FirstCond2}
L(1,\chi_d)= \prod_{p\leq y}\left(1-\frac{\chi_d(p)}{p}\right)^{-1}\exp\left(\sum_{y<p<z}\frac{\chi_d(p)}{p} +O\left(\frac{1}{\sqrt{y}\log y}\right)\right)..
\end{equation}
Furthermore,  taking $k=[\log x/(50\log\log x)]$ in Proposition \ref{LargeMoments} implies that the number of fundamental discriminants $d\in \mathcal{D}(x)$ such that 
$$ \left|\sum_{y<p<z}\frac{\chi_d(p)}{p}\right|>\frac{1}{(\log\log x)^{1/4}}$$ is 
\begin{equation}\label{SecondCond}
\ll \sqrt{x} \left(\frac{\log x}{y\log y(\log\log x)^{1/3}}\right)^{k}.
 \end{equation}
On the other hand, it follows from Corollary  \ref{MoWe} that there are at least $\sqrt{x}e^{-y(1+o(1))}$ fundamental discriminants  $d\in \mathcal{D}(x)$ for which $\chi_d(p)=1$ for all primes $p\leq y$. Therefore, choosing $y= \log x/(2\log\log x)$ we deduce from \eqref{FirstCond2} and \eqref{SecondCond} that there are at least $x^{1/2-1/\log\log x}$ fundamental discriminants  $d\in \mathcal{D}(x)$ such that $\chi_d(p)=1$ for all primes $p\leq y$,  \eqref{FirstCond2} holds and 
$$ \left|\sum_{y<p<z}\frac{\chi_d(p)}{p}\right|\leq \frac{1}{(\log\log x)^{1/4}}.$$ For these $d$, we have by \eqref{FirstCond2} that 
 $$
 L(1,\chi_d)=e^{\gamma} \log\log x \left(1+O\left(\frac{1}{(\log \log x)^{1/4}}\right)\right).
 $$
 Inserting this estimate in \eqref{ClassSpecial} completes the proof. 

\end{proof}

\section{An upper bound for the number of real quadratic fields with extreme class number: proof of Theorem \ref{Main}, part (b)}
In order to bound the number of real quadratic fields $\mathbb{Q}(\sqrt{d})$ with discriminant $d\leq x$ for which the class number $h(d)$ is extremely large (that is, $h(d)$ satisfies \eqref{Extreme}), we shall first bound the number of   small solutions to the Pell equations $m^2-dn^2=\pm 4$. We prove the following lemma. 

\begin{lem}\label{PellCount}
Let $x$ be large,  and let $d\leq x$ be a positive integer. For a real number $\theta\in (1/2, 3/2)$ denote by $S_{\theta}(d)$ the set of positive solutions $(m,n)$ to the Pell equations 
\begin{equation}\label{Pell}
m^2-dn^2=\pm 4,
\end{equation}
such that  $m\leq d^{\theta}.$ Then, we have 
$$ \sum_{d\leq x}|S_{\theta}(d)|\ll \big(x^{1/2}+x^{\theta-1/2}\big)(\log x)^2.$$
\end{lem}

\begin{proof} Let  $P_{\theta}(d)$  be the set of solutions $(m,n)$ to the positive Pell equation
$m^2-dn^2= 4,$ such that  $m\leq d^{\theta}.$ Similarly, let $N_{\theta}(d)$  be the set of solutions $(m,n)$ to the negative Pell equation
$m^2-dn^2= -4,$ such that  $m\leq d^{\theta}.$ We shall only bound $\sum_{d\leq x}|P_{\theta}(d)|$ since the treatment for $\sum_{d\leq x}|N_{\theta}(d)|$ is similar. Let $(m,n)\in P_{\theta}(d)$. Then, note that 
$dn^2\leq m^2\leq d^{2\theta}$, and hence $n\leq d^{\theta-1/2}$. Furthermore, for a fixed $n\leq x^{\theta-1/2}$, if $(m,n)\in P_{\theta}(d)$ for some $d\leq x$  then $m\leq n\sqrt{x}+2$ and $m^2\equiv 4 \pmod {n^2}$. Therefore, we deduce that 
$$ \sum_{d\leq x}|P_{\theta}(d)|\leq \sum_{n\leq x^{\theta-1/2}}|\{m\leq n\sqrt{x}+2, \text{ such that } m^2
\equiv 4 \pmod {n^2}\}|.$$
 Let $\ell(q)$ be the number of solutions $m \pmod q$ of the congruence $m^2\equiv 4 \pmod q$. Then, 
$\ell(q)$ is a multiplicative function, and
$$ \ell(p^k)\leq \begin{cases} 2  &\text{ if } p > 2,\\ 4 &\text{ if } p=2. \end{cases}$$
Hence, we derive
\begin{align*}
&\sum_{n\leq x^{\theta-1/2}} |\{m\leq n\sqrt{x}+2, \text{ such that } m^2
\equiv 4 \pmod {n^2}\}| \\
& \ll \sum_{n\leq x^{\theta-1/2}}\ell(n^2)\left( \frac{\sqrt{x}}{n}+1\right)\ll 
\big(x^{1/2}+x^{\theta-1/2}\big)  \sum_{n\leq x^{\theta-1/2}}\frac{\ell(n^2)}{n}.
\end{align*}
The lemma follows upon noting that 
$$ 
\sum_{n\leq x^{\theta-1/2}}\frac{\ell(n^2)}{n} \ll \prod_{2<p<x} \left(1-\frac{1}{p}\right)^{-2}\ll (\log x)^2.
$$
\end{proof}
The second ingredient in the proof of part (b) of Theorem 1.2 is to show that $L(1,\chi_d)\leq (2e^{\gamma}+o(1))\log\log d$ for all but at most $x^{1/2+o(1)}$ fundamental discriminants $0< d<x$. To this end, 
we shall use Heath-Brown's quadratic large sieve to show that Proposition \ref{GranvilleSound} can be improved if we restrict our attention to quadratic characters. More precisely, we prove that $L(1,\chi_d)$ can be approximated by an Euler product over the primes $p\leq (\log x)^A$, for all but at most $x^{1/A+o(1)}$ fundamental discriminants $0< d<x$.

\begin{pro}\label{AlmostAllShort}
Let $A>1$ be fixed. Then for all but at most $x^{1/A+o(1)}$ fundamental discriminants $0<d<x$ we have 
$$ 
L(1,\chi_d)=\prod_{p\leq (\log x)^A}\left(1-\frac{\chi_d(p)}{p}\right)^{-1}\left(1+O\left(\frac{1}{\log\log x}\right)\right).
$$
\end{pro}

\begin{proof} First, it follows from Proposition \ref{GranvilleSound} that 
$$ L(1,\chi_d)= \prod_{p\leq (\log x)^{2A}}\left(1-\frac{\chi_d(p)}{p}\right)^{-1}\left(1+O\left(\frac{1}{\log\log x}\right)\right),$$
for all except at most $x^{1/A+o(1)}$ fundamental discriminants $0<d<x$. To prove the result, we are going to show that
\begin{equation}\label{CrucialBound}
 \sum_{(\log x)^A<p< (\log x)^{2A}}\frac{\chi_d(p)}{p}=O\left(\frac{1}{\log\log x}\right),
 \end{equation}
 for all but at most $x^{1/A+o(1)}$ fundamental discriminants $0<d<x$. To this end, we will exploit Heath-Brown's quadratic large sieve (see Corollary 2 of \cite{HB}) which asserts that 
\begin{equation}\label{HeathBrownLarge}
 \sideset {}{^\flat}\sum_{0<d< x} \left| \sum_{n\leq N} a(n)\chi_d(n)\right|^2
\ll_{\epsilon} (xN)^{\epsilon}(x+N)\sum_{\substack{n_1, n_2\leq N\\ n_1n_2=\square}}|a(n_1)a(n_2)|.
\end{equation}
where the $\sideset {}{^\flat}\sum$ is taken over fundamental discriminants, and the $a(n)$ are arbitrary complex numbers.

For $0\leq j\leq J:=[A\log\log x/\log 2]$, we define $z_j=2^j(\log x)^A$, and put $z_{J+1}=(\log x)^{2A}$. Also, we let $k=[(\log x)/(A\log\log x)]+1$ so that $z_j^k\geq x$ for all $0\leq j\leq J+1$.  Now, similarly to \eqref{Expansion} we have 
$$
\left(\sum_{z_j<p<z_{j+1}}\frac{\chi_d(p)}{p}\right)^k=\sum_{z_j^k<m<z_{j+1}^k}\frac{b_{k}(m;z_j,z_{j+1})\chi_d(m)}{m},
$$
where the coefficient $b_{k}(m;z_j,z_{j+1})$ is defined in \eqref{Combinatorics}.
 Then, by \eqref{HeathBrownLarge} we obtain
\begin{equation}\label{Heath-BrownSieve}
\begin{aligned}
 \sideset {}{^\flat}\sum_{0<d< x}\left(\sum_{z_j<p<z_{j+1}}\frac{\chi_d(p)}{p}\right)^{2k}
&=  \sideset {}{^\flat}\sum_{0<d< x} \left(\sum_{z_j^k<n<z_{j+1}^k}\frac{b_{k}(n;z_j,z_{j+1})\chi_d(n)}{n}\right)^2\\
& \ll_{\epsilon}(z_{j+1})^{k(1+\epsilon)}\sum_{\substack{z_j^k<n, m<z_{j+1}^k\\ mn=\square}} \frac{b_{k}(m;z_j,z_{j+1})b_{k}(n;z_j,z_{j+1})}{mn}.
\end{aligned}
\end{equation}
Let $m$ and $n$ be positive integers such that $\Omega(m)=\Omega(n)=k$ and put $d=(m,n)$. Also, put $n=dn_1$ and $m=dm_1$ where $(m_1,n_1)=1.$ Since $mn=\square$ then both $m_1$ and $n_1$ are squares.  Let $n_1=\ell_1^2$ and $m_1=\ell_2^2$, and put $s=\Omega(\ell_1)$. Since $\Omega(n)=\Omega(m)=k$, then $\Omega(\ell_2)=s$ and $\Omega(d)=k-2s$. Therefore, using \eqref{CombInequality} we obtain
\begin{align*}
&b_{k}(n;z_j,z_{j+1})b_{k}(m;z_j,z_{j+1})\\
&\leq \binom{k}{2s}^2
b_{2s}(\ell_1^2;z_j,z_{j+1})b_{2s}(\ell_2^2;z_j,z_{j+1})\big(b_{k-2s}(d;z_j,z_{j+1})\big)^2\\
&\leq \binom{k}{2s}^2\binom{2s}{s}^2\Big(b_{s}(\ell_1;z_j,z_{j+1})b_{s}(\ell_2;z_j,z_{j+1})b_{k-2s}(d;z_j,z_{j+1})\Big)^2\\
& \leq k! \binom{k}{s, s, k-2s} b_{s}(\ell_1;z_j,z_{j+1})b_{s}(\ell_2;z_j,z_{j+1})b_{k-2s}(d;z_j,z_{j+1}),
\end{align*}
since $b_r(e;z_j,z_{j+1})\leq r!$ for any positive integers $r$ and $e$. 
Thus, we deduce
\begin{align*}
&\sum_{\substack{z_j^k<n, m<z_{j+1}^k\\ mn=\square}}\frac{b_{k}(n;z_j,z_{j+1})b_{k}(m;z_j,z_{j+1})}{mn}\\
&\leq  k! \sum_{0\leq s\leq k/2}\binom{k}{s, s, k-2s}\left(\sum_{d}\frac{b_{k-2s}(d;z_j,z_{j+1})}{d^2}\right) \left(\sum_{\ell} \frac{b_{s}(\ell;z_j,z_{j+1})}{\ell^2}\right)^2\\
&\leq  3^k k!\left(\sum_{z_j<p<z_{j+1}}\frac{1}{p^2}\right)^{k},
\end{align*}
by \eqref{PrimeSumPower}. Inserting this bound in \eqref{Heath-BrownSieve}  and using that $\sum_{p>z_j}1/p^2\ll 1/(z_j\log z_j)$ yields
$$\sideset {}{^\flat}\sum_{0<d< x}\left(\sum_{z_j<p<z_{j+1}}\frac{\chi_d(p)}{p}\right)^{2k}
\ll_{\epsilon}\left(z_{j+1}^{\epsilon} k\right)^k\ll_{\epsilon}\left((\log x)^{2\epsilon A}k\right)^k.$$ 
Therefore, the number of fundamental discriminants $0<d<x$ such that 
$$\sum_{z_j<p<z_{j+1}}\frac{\chi_d(p)}{p}>\frac{1}{A(\log\log x)^2},$$
is 
$$
\ll_{\epsilon} \Big(A^2(\log\log x)^4(\log x)^{2\epsilon A}k\Big)^k\ll_{\epsilon}  x^{1/A+ 3\epsilon}.
$$
Thus, we deduce that \eqref{CrucialBound} holds for all but at most $O_{\epsilon}(x^{1/A+ 3\epsilon}\log\log x)$ fundamental discriminants $0<d<x$, as desired.

\end{proof}

\begin{proof}[Proof of Theorem \ref{Main}, part (b)]
Taking $A=2$ in Proposition \ref{AlmostAllShort}, we deduce that for all but at most $x^{1/2+o(1)}$ fundamental discriminants $x^{1/4}<d<x$, we have 
\begin{equation}\label{ALMOSTBoundL1}
L(1,\chi_d)\leq (2 e^{\gamma} +o(1))\log\log d.
\end{equation}
Now, the class number formula \eqref{ClassReal} implies that if $d$ satisfies \eqref{ALMOSTBoundL1} then
$$ h(d)\leq (2e^{\gamma}+o(1))\sqrt{d}\cdot\frac{\log\log d}{\log \varepsilon_d}.$$
Therefore, if $d$ is a fundamental discriminant such that $x^{1/4}<d<x$, $L(1, \chi_d)$ satisfies $\eqref{ALMOSTBoundL1}$ and $h(d)$ satisfies the bound \eqref{Extreme}, then 
$ \varepsilon_d\leq d^{1+o(1)}.$ 
Now, recall that 
$\varepsilon_d= (m+n\sqrt{d})/2>m/2,$
where $(m,n)$ is a solution to the Pell equations \eqref{Pell}. Thus, we deduce from Lemma \ref{PellCount} that the number of these fundamental discriminants is at most 
 $x^{1/2+o(1)}$, which completes the proof.
\end{proof}

\section*{Aknowledgements}

I would like to thank Andrew Granville for an interesting suggestion and for Lemma 3.1. I am also grateful to Igor Shparlinski for helpful discussions. I would also like to thank the anonymous referee for a careful reading of the paper and for a useful suggestion which removed the assumption of GRH from the statement of Theorem 1.2, part (b). The author is partially supported by a Discovery Grant from the Natural Sciences and Engineering Research Council of Canada.


\begin{thebibliography}{DDDD}

\bibitem[1] {Ch} S. Chowla, 
\emph{Improvement of a theorem of Linnik and Walfisz.}
 Proc. London Math. Soc. 50 (1949), 423--429.

\bibitem[2]{Du1} W. Duke, 
\emph{Extreme values of Artin L-functions and class numbers.} 
Compositio Math. 136 (2003), no. 1, 103--115.

\bibitem[3]{Du2} W. Duke, 
\emph{Number fields with large class group.} Number theory, 117--126, CRM Proc. Lecture Notes, 36, Amer. Math. Soc., Providence, RI, 2004. 

\bibitem[4] {GrSo1} A. Granville and K. Soundararajan,
\emph{Large character sums.} 
 J. Amer. Math. Soc. 14 (2001), no. 2, 365--397.


\bibitem[5] {GrSo2} A. Granville and K. Soundararajan,
\emph{The distribution of values of $L(1, \chi_d)$.} 
Geom. Funct. Anal. 13 (2003), no. 5, 992--1028. 


\bibitem[6] {HB} D. R. Heath-Brown, 
\emph{A mean value estimate for real character sums.}
Acta Arith. 72 (1995), no. 3, 235--275. 

\bibitem[7]{Li}  J. E. Littlewood,
\emph{On the class number of the corpus $P(\sqrt{-k})$.} 
Proc. London Math. Soc. 27 (1928), 358--372.

\bibitem[8] {MoWe} H. L. Montgomery and  J. P. Weinberger, 
\emph{Real quadratic fields with large class number.} 
Math. Ann. 225 (1977), no. 2, 173--176. 

\bibitem[9] {St} T. Storer,
\emph{Cyclotomy and difference sets.} 
Lectures in Advanced Mathematics, No. 2 Markham Publishing Co., Chicago, Ill. 1967 vii+134 pp.




\end{thebibliography}
\end{document}